\documentclass{amsart}
\newcommand{\triplearrows}{\begin{smallmatrix} \to \\ \to \\ 
\to \end{smallmatrix} }
\newcommand{\perf}{\mathrm{Perf}}

\usepackage{url}
\usepackage{verbatim}

\renewcommand{\sp}{\mathrm{Sp}}

\renewcommand{\phi}{\varphi}
\usepackage{relsize} 
\usepackage[bbgreekl]{mathbbol} 
\usepackage{amsfonts} 
\DeclareSymbolFontAlphabet{\mathbb}{AMSb} 
\DeclareSymbolFontAlphabet{\mathbbl}{bbold} 
\newcommand{\Prism}{{\mathlarger{\mathbbl{\Delta}}}}
\usepackage{geometry}
\usepackage{amsfonts}
\usepackage{amssymb}
\usepackage{amsmath}
\usepackage{amsthm}
\usepackage{cleveref}
\renewcommand{\hom}{\mathrm{Hom}}
\newcommand{\THH}{\mathrm{THH}}
\newcommand{\TC}{\mathrm{TC}}
\newcommand{\TP}{\mathrm{TP}}
\newcommand{\ainf}{A_{\mathrm{inf}}}

\newcommand{\md}{\mathrm{Mod}}

\usepackage{xy}
\input xy

\xyoption{all}

\theoremstyle{definition}
\newtheorem{definition}{Definition}[section]
\newtheorem{example}[definition]{Example}
\newtheorem{construction}[definition]{Construction}
\newtheorem{remark}[definition]{Remark}
\newtheorem{question}[definition]{Question}
\theoremstyle{theorem}
\newtheorem{proposition}[definition]{Proposition}
\newtheorem{lemma}[definition]{Lemma}
\newtheorem{corollary}[definition]{Corollary}

\newtheorem{theorem}[definition]{Theorem}

\begin{document}

\newcommand{\galq}{\mathrm{Gal}_{\mathbb{Q}}}
\newcommand{\st}{\mathrm{St}}

\title{Remarks on $K(1)$-local $K$-theory}
\author{Bhargav Bhatt, Dustin Clausen, and Akhil Mathew}
\date{\today}
\maketitle 

\begin{abstract}
We prove two basic structural properties of the algebraic $K$-theory of rings
after $K(1)$-localization at an implicit prime $p$. Our first result (also
recently obtained by Land--Meier--Tamme by different methods) states that
$L_{K(1)} K(R)$ is insensitive to inverting $p$ on $R$; we deduce this from
recent advances in prismatic cohomology and $\mathrm{TC}$. Our second result
yields a K\"unneth formula in $K(1)$-local $K$-theory for adding $p$-power roots of unity to $R$. 
\end{abstract}
\section{Introduction}

In this note, we consider the algebraic $K$-theory spectrum $K(R)$ of a ring
$R$, after applying the operation $L_{K(1)}$ of $K(1)$-localization at a prime
$p$ which is fixed throughout.
The construction $R \mapsto L_{K(1)} K(R)$ featured in the work of Thomason
\cite{Tho85} connecting algebraic $K$-theory and \'{e}tale cohomology, cf.~\cite{Mitchell} for a survey. Here we record two basic structural
features of $L_{K(1)} K(R)$. 

We first show that $K(1)$-local $K$-theory is insensitive to
inverting $p$; a stronger result  (for connective $K(1)$-acyclic
$E_1$-rings) has been obtained recently by  
Land--Meier--Tamme in \cite[Cor.~3.28]{LMT}. 

\begin{theorem}\label{mainthmintro}
Let $A$ be an associative ring, or even an $E_1$-algebra over $\mathbb{Z}$. Then the map of spectra $K(A) \to K(A[1/p])$ induces an equivalence
$L_{K(1)} K(A) \simeq L_{K(1)} K(A[1/p])$. 
\end{theorem}

\begin{example}[$p$-power torsion rings] 
\label{ptorsionring}
When $A$ is $p$-power torsion, we conclude that $L_{K(1)} K(A)  = 0$. 
When $A$ is simple $p$-torsion (i.e., an $\mathbb{F}_p$-algebra), this follows
from Quillen's calculation \cite{Qui72} of the $K$-theory of finite fields, in
particular that $K(\mathbb{F}_p; \mathbb{Z}_p) \simeq
H\mathbb{Z}_p$. 
However, for $\mathbb{Z}/p^n$, one knows the  $p$-adic
$K$-theory only in a certain range \cite{Angeltveit, Brun}, so it seems
difficult to prove the result by direct computation. \end{example} 

In \cite{LMT}, Land--Meier--Tamme
give a purely homotopy-theoretic proof
of 
the result of \Cref{ptorsionring}, applying more generally to certain ring
spectra; from this \Cref{mainthmintro} is a consequence. 

Our first goal is to give an arithmetic proof of \Cref{mainthmintro}, as a
$K$-theoretic version of the \'etale comparison theorem of
\cite[Th.~9.1]{Prisms}. 
In fact, the assertion $L_{K(1)}K(\mathbb{Z}/p^n)=0$ is a
quick consequence of  recent advances in topological cyclic homology \cite{BMS2} and the theory of
prismatic cohomology \cite{Prisms}. While we do not know the $K$-theory of $\mathbb{Z}/p^n$, the work
\cite{BMS2, CMM, Prisms} leads to a relatively explicit calculation of the $K$-theory of
$\mathcal{O}_C/p^n$ via $\TC$, for $C$ the completed algebraic closure of
$\mathbb{Q}_p$ and $\mathcal{O}_C \subset C$ the ring of integers.  We can calculate directly there that the
Bott element is $p$-adically nilpotent, and then we use \cite{CMNN} to descend. 

In fact, we can obtain (via \cite{CMM}) the following consequence, which is a
$K$-theoretic version of the \'etale comparison theorem: 

\begin{corollary} 
Let $R$ be any commutative ring which is henselian along $(p)$. Then there is a natural
equivalence
$L_{K(1)} \TC(R) \simeq L_{K(1)} K(R[1/p])$. 
\end{corollary}

Our second result 
is a type of K\"unneth formula in $K(1)$-local $K$-theory. 
In general, $K$-theory does not satisfy a K\"unneth formula: it is only a lax
symmetric monoidal, not a symmetric  monoidal functor. 
Here we show that in the special case of adding $p$-power roots of unity, 
one does have a K\"unneth formula which one can make explicit. 

To formulate the result, we recall that $\mathbb{Z}_p^{\times}$ naturally acts both on
$\mathbb{Z}[\zeta_{p^\infty}]$  and on the $p$-complete $E_\infty$-ring
$KU_{\hat{p}}$, by Galois automorphisms and Adams operations respectively. 
For a ring $R$, we write $R[\zeta_{p^\infty}] = R \otimes_{\mathbb{Z}}
\mathbb{Z}[\zeta_{p^\infty}]$. 
\begin{theorem}
\label{Kformintro}
Let $R$ be a commutative ring. 
Then there are natural, $\mathbb{Z}_p^{\times}$-equivariant equivalences of
 $E_\infty$-rings 
$$L_{K(1)} K ( R[\zeta_{p^\infty}]) \simeq (K(R) \otimes
KU_{\hat{p}})_{\hat{p}}  .$$
\end{theorem}

\Cref{Kformintro} is related to results of Dwyer--Mitchell \cite{DM98} and
Mitchell \cite{Mit00}; our construction of the comparison map is based on the
description of Snaith \cite{Snaith} of $KU$. 
Furthermore, one can obtain an analog of this formula for any localizing invariant over
$\mathbb{Z}[1/p]$-algebras which commutes
with filtered colimits. 
Using these ideas, we also give a complete description of  $K(1)$-local
$K$-theory as an \'etale sheaf of spectra 
on $\mathbb{Z}[1/p]$-algebras (under appropriate finiteness conditions),
cf.~\Cref{K1localKthy}, yielding a spectrum-level version
of Thomason's spectral sequence from \cite{Tho85}. 

\subsection*{Acknowledgments}
We thank Lars Hesselholt, Jacob Lurie, and Peter Scholze for helpful discussions. 
We are also grateful to the anonymous referee for their helpful comments on the first version of this paper. 
The third author would like to thank the University of Copenhagen for its
hospitality during which some of this work was done. 
The first author was supported by NSF grant \#1801689 as well as grants from the Packard and Simons Foundations.  The second author was supported by Lars Hesselholt's Niels Bohr professorship, the University of Bonn, and the Max Planck Institute for Mathematics.
This work was done while the third author was a Clay Research Fellow. 

\section{Proof of \Cref{mainthmintro}}
\subsection{$\delta$-ring calculations}

In this section, we prove  a simple nilpotence result
(\Cref{nilpotencelemma}).\footnote{As in \Cref{useofetalecomp} below, one could
replace its use below with that of the \'etale comparison theorem of
\cite{Prisms}.} 
We freely use the theory of $\delta$-rings introduced in
\cite{Joy85}.\footnote{$\delta$-rings also arise as the natural
structure on the homotopy groups of $K(1)$-local $E_\infty$-ring spectra (where
they are often called $\theta$-algebras or Frobenius algebras), cf.~\cite{Ho14}. We will not use this fact here.} 
Given a $\delta$-ring $(R, \delta)$, we let $\phi: R \to R$ be the map $\phi(x)
= x^p + p \delta(x)$, so that $\phi$ is a ring homomorphism. 
We recall the basic formulas 
\begin{gather} \delta(ab) = a^p \delta(b) + b^p \delta(a) + p \delta(a) \delta(b)  = 
\phi(a) \delta(b) + \delta(a)b^p  , \label{ea} \\
\delta(a + b) = \delta(a) + \delta(b) - \sum_{0 < i < p}  \frac{1}{p}
\binom{p}{i} a^i b^{p-i} \label{eb}
\end{gather} for $a, b
\in R$. 

Let $R$ be a $p$-complete $\delta$-ring. 
In \cite[Def.~2.19]{Prisms}, the crucial notion of a distinguished element
is introduced: an
element $x \in R$ is called \emph{distinguished} if $\delta(x)$ is a unit. 
For example, the element $p$ is always distinguished. 
Here we use the following generalization. 

\begin{definition} 
An element $x$ of a $p$-complete $\delta$-ring $R$ is called
\emph{weakly $k$-distinguished} if
$(x, \delta(x), \dots, \delta^k(x))$ is the unit ideal.
\end{definition} 
\begin{example} 
The element $p^k$ is weakly $k$-distinguished  in any $p$-complete $\delta$-ring. 
It suffices to check this in $\mathbb{Z}_{p}$. 
Indeed, the formula $\delta(x) = \frac{x - x^p}{p}$ (valid for $x
\in \mathbb{Z}_p$) shows easily that 
if the $p$-adic valuation $v_p(x)$ is positive, then $v_p(\delta(x)) = v_p(x) -
1$. Inductively, we thus get that $v_p( \delta^k( p^k)) = 0$, so $p^k$ is
weakly $k$-distinguished. 
\end{example} 

\begin{definition} 
Let $R$ be a $\delta$-ring. 
Let $I \subset R$ be an ideal. We define $\delta(I)$ as the ideal generated by
$\delta(x), x \in I$. 
\end{definition} 

\begin{example} 
Suppose $I = (x)$. Then $\delta(I) \subset (x, \delta(x))$. 
More generally, if $I \subset R$ is an ideal generated by elements $(f_1, \dots, f_n)$, then 
\begin{equation}  \label{deltaoffgideal} \delta(I) \subset (f_1, \dots, f_n, \delta(f_1),\dots,
\delta(f_n)).  \end{equation}
This follows easily from the formulas \eqref{ea} and  \eqref{eb} above. 
\end{example} 

\begin{proposition}[Nilpotence criterion]
\label{nilpotencelemma}
Let $R$ be a $\delta$-ring, and let $x,y \in R$. Suppose $R$ is $(p, x)$-adically
complete and we have the equation $x y = p^k$. Then $y$ is weakly
$(k-1)$-distinguished and $x$ is $p$-adically
nilpotent.  
\end{proposition} 

\begin{proof}
We first claim that $y$ is weakly $(k-1)$-distinguished. 
Indeed, consider the ideal $(p^k) = (xy)$. 
We claim that for each $i \geq 1$, we have that 
\begin{equation}  \delta^i ( p^k) \in 
( \phi^i(x) \delta^i(y), \delta^{i-1}(y), \dots, y). \label{deltai}
\end{equation}
To see this, we use induction on $i$. 
For $i = 1$, we have
$\delta(xy) = \phi(x) \delta(y) + \delta(x) y^p$, as desired. 
If we have proven \eqref{deltai} for a given  $i$, then we can apply 
$\delta$ to both sides and use \eqref{deltaoffgideal}
to conclude the result for $i + 1$, together with $\delta( \phi^i(x) \delta^i(y))
= \phi^{i+1}(x) \delta^{i+1}(y) + \delta(\phi^i(x)) \delta^i(y)^p$. 
By induction on $i$, this proves 
\eqref{deltai} in general. 

Taking $i = k$ in \eqref{deltai} and using that $\delta^k(p^k)$ is a unit, we find that  
$\phi^k(x) \delta^k (y), \delta^{k-1}(y), \dots, y$ generate the unit ideal in
$R$. 
But since $\phi^k(x) $ is contained in the Jacobson radical of
$R$ (as $R$ is $(p, x)$-adically
complete and $\phi^k(x) \equiv x^{p^k}$ modulo $p$),  we conclude that $\delta^{k-1}(y), \dots, y$ generate the unit ideal of
$R$. Thus, $y$ is weakly $(k-1)$-distinguished. 

Finally, we must show that $x$ is $p$-adically nilpotent. 
Consider the $p$-adic completion $R'$ of $R[1/x]$; this is also a $p$-complete
$\delta$-ring, and it suffices to show that $R' = 0$. But the image of $y$ in
$R'$ is both a unit multiple of $p^k$ and weakly $(k-1)$-distinguished, so the
ideal $(y, \delta(y), \dots, \delta^{k-1}(y)) $ is both contained in $(p)$ and the unit ideal. 
This now shows that $R' = 0$ as desired. 
\end{proof} 
\subsection{The vanishing result for $L_{K(1)}\TP(\mathcal{O}_C/p^n)$}
In this subsection, we 
let $C$ be the completion of the algebraic closure of $\mathbb{Q}_p$, let $\mathcal{O}_C$ be its ring of integers, and let $\ainf$ denote Fontaine's period ring, with its canonical surjective map $\theta: \ainf \to \mathcal{O}_C$.  The kernel
of $\theta$ is generated by a nonzerodivisor, a choice of which we denote $d$. With respect to the unique $\delta$-structure
on $\ainf$, $d$ is a distinguished element and
$(\ainf, (d))$ is the perfect prism corresponding to the integral perfectoid
ring $\mathcal{O}_C$, \cite[Th.~3.10]{Prisms} and \cite[Sec.~3]{BMS1}. 

We can fix such a $d$ as follows. Consider a system $(1, \zeta_p, \zeta_{p^2}, \dots)$ of compatible primitive $p$-power roots of unity in $\mathcal{O}_C$ and let $\epsilon$ denote the corresponding element in $\mathcal{O}_C^\flat = \varprojlim_{\operatorname{Frob}} \mathcal{O}_C/p$.  Then we can take $d$ to be the element
$$d = \frac{[\epsilon] - 1}{[\epsilon^{1/p}] - 1} \in \ainf = W(\mathcal{O}_C^\flat).$$
It is well-known that this choice of $d$ generates the kernel of $\theta$.   See \cite[Sec. 3]{BMS1} for a treatment of all of these facts. 

Next we recall the calculation of topological Hochschild invariants of
$\mathcal{O}_C$, using the notation and language of \cite{NS18}. 
\begin{proposition}[{Hesselholt \cite{Hes06}, Bhatt--Morrow--Scholze \cite[Sec. 6]{BMS2}}] 
We can choose isomorphisms \[ \TC^-(\mathcal{O}_C; \mathbb{Z}_p) \simeq \ainf[u, v]/(uv - d), \quad 
\TP(\mathcal{O}_C; \mathbb{Z}_p) \simeq \ainf[\sigma^{\pm 1}], \quad |u| = 2,
|v| = -2, |\sigma| = 2,
\]
such that the canonical map is the identity on $\ainf$ and carries 
$v \mapsto \sigma^{-1}, u \mapsto d \cdot \sigma$ and the cyclotomic Frobenius map is the
Frobenius on $\ainf$ and carries $u \mapsto \sigma$. 
\end{proposition} 

\begin{remark} 
In degree zero, the above isomorphisms are canonical. 
However, in nonzero degrees, they are not canonical; for
example, they are not Galois-equivariant.  The canonical form of the above
proposition involves the so-called Breuil--Kisin twists as in \cite{BMS2}. 
\end{remark} 

\begin{construction}[$K(1)$-localization explicitly]
Recall from \cite[Lemma 3.1]{Niziol} or \cite[Lemma 1.3.7]{HN} that the
localization sequence shows $K( \mathcal{O}_C;
\mathbb{Z}_p)\overset{\sim}{\rightarrow} K(C;\mathbb{Z}_p)$, and Suslin's
rigidity theorem \cite{Suslin} shows that the latter is isomorphic to
$ku_{\widehat{p}}$ (i.e., $p$-complete connective topological $K$-theory) as a
ring spectrum by choosing any ring isomorphism $C\cong \mathbb{C}$.  The
$K(1)$-localization of $ku$ is implemented by inverting the generator in degree
2 and then $p$-completing, as is clear from the definition.  It follows that the
$K(1)$-localization of $K( \mathcal{O}_C; \mathbb{Z}_p)$, or more generally of
any $p$-complete $K(\mathcal{O}_C; \mathbb{Z}_p)$-module $M$, can be obtained in the analogous way:
$$L_{K(1)}M = M[\beta^{-1}]_{\widehat{p}},$$
where $\beta\in \pi_2K(\mathcal{O}_C;\mathbb{Z}_p)\cong \mathbb{Z}_p$ is any generator.
\end{construction}

Next we trace this into $\TP$, where  
one can identify the image of the cyclotomic trace. 
\begin{proposition}[{\cite[Th.~1.3.6]{HN}}] 
With respect to the above identifications, the cyclotomic trace
$K_*(\mathcal{O}_C; \mathbb{Z}_p) \to \TP_*(\mathcal{O}_C; \mathbb{Z}_p)$ carries $\beta$ to a $\mathbb{Z}_p^{\times}$-multiple of
$([\epsilon] - 1) \sigma$.
\end{proposition} 

Let $R$ be a quasiregular semiperfectoid $\mathcal{O}_C$-algebra (in the sense
of \cite[Sec.~4]{BMS2}), e.g., the quotient of a perfectoid by a regular
sequence. 
Then one can construct  \cite[Sec.~7]{Prisms}
a $(p, d)$-adically complete and $d$-torsion-free $\delta$-ring $\Prism_R$, which receives a canonical map from $\ainf$, 
and a map 
$R \to \Prism_R/(d)$; moreover, $\Prism_R$ is universal for this structure. 
The ring $\Prism_R$ is equipped with the Nygaard filtration (also defined in loc.~cit.) 
whose completion is denoted $\widehat{\Prism}_R$, and acquires a
$\delta$-structure itself. 
Our primary tool in this paper, which connects algebraic $K$-theory (or rather 
$\TP$) and $\delta$-rings, is the following result. 

\begin{theorem}[{\cite{BMS2} and  \cite[Sec.~13]{Prisms}}] 
For a quasiregular semiperfectoid $\mathcal{O}_C$-algebra $R$, 
$\TP_*(R; \mathbb{Z}_p)$ is concentrated in even degrees, is 2-periodic, and 
there is a canonical isomorphism
$\pi_0 \TP(R; \mathbb{Z}_p) \simeq \widehat{\Prism}_R$. 
\end{theorem} 

Using this, we can give a 
direct description of the $K(1)$-localization of $\TP$ in terms of
$\widehat{\Prism}$. 
\begin{corollary} 
\label{K1localTP}
For a quasiregular semiperfectoid $\mathcal{O}_C$-algebra $R$, 
there is a canonical isomorphism
$\pi_0 (L_{K(1)}\TP(R)) \simeq (\widehat{\Prism}_R[1/d])_{\hat{p}}$. 
\end{corollary} 
\begin{proof} 
The spectrum $L_{K(1)} \TP(R)$ is obtained by inverting (in the $p$-complete
category) the image of the Bott
element from $K_*(\mathcal{O}_C; \mathbb{Z}_p)$ via the trace map. 
As we saw, the map 
$K_*(\mathcal{O}_C; \mathbb{Z}_p) \to \TP_*(\mathcal{O}_C; \mathbb{Z}_p)$
carries the class of $\beta$ to a graded unit times the class of $[\epsilon] - 1
\in \ainf$.  However, in $\ainf$ we have $[\epsilon] - 1 \equiv ([\epsilon^{1/p}] - 1)^p$
(modulo $p$) and $d \equiv ([\epsilon^{1/p}] - 1)^{p-1}$ (modulo $p$); thus,
inverting either $[\epsilon] -1 $ or $d$ in the $p$-complete sense is the same
operation, completing the proof.
\end{proof} 

Finally, we can conclude the main 
vanishing result that was the goal of this section. 

\begin{corollary} 
\label{K1TPvanishes}
For each $n$, we have that $L_{K(1)} (\TP( \mathcal{O}_C/p^n)) = 0$. 
\end{corollary} 
\begin{proof} 
As there is a ring map $\Prism_{\mathcal{O}_C/p^n}\rightarrow \widehat{\Prism}_{\mathcal{O}_C/p^n}$, by the above it suffices to show that $d$ is $p$-adically nilpotent in $\Prism_{\mathcal{O}_C/p^n}$.  But by definition $\Prism_{\mathcal{O}_C/p^n}$ is a $(p,d)$-adically complete $\delta$-ring such that there is a homomorphism $\mathcal{O}_C/p^n\rightarrow \Prism_{\mathcal{O}_C/p^n}/d$.  It follows
that we can solve the equation 
$d y = p^n$ in $\Prism_{\mathcal{O}_C/p^n}$, and we deduce that $d$ is $p$-adically
nilpotent by 
\Cref{nilpotencelemma}, as desired. 
\end{proof}

\begin{remark} 
\label{useofetalecomp}
The main result that was shown above is that if $R$ is a 
$p$-power torsion $\mathcal{O}_C$-algebra which is quasiregular semiperfectoid,
then $d$ is $p$-adically nilpotent in $\Prism_R$. This is a special
case of the \'etale comparison theorem \cite[Theorem 9.1]{Prisms}, since in
this case the
generic fiber of $\mathrm{Spf}(R)$ vanishes; in particular, the use of the
\'etale comparison theorem could replace \Cref{nilpotencelemma} above. 
\end{remark}

\subsection{The $K(1)$-local $K$-theory of $\mathbb{Z}/p^n$}

Here we prove the following special case of our main result. 
\begin{proposition} 
\label{Zpn}
For each $n$, we have $L_{K(1)} K(\mathbb{Z}/p^n) = 0$. 
\end{proposition} 
\begin{proof} 
We first 
prove the weaker assertion that if $C$ is as in the previous section, then 
$L_{K(1)} K(\mathcal{O}_C/p^n) = 0$. 
Indeed, by \cite[Th.~C]{CMM}, 
the cyclotomic trace 
$K(\mathcal{O}_C/p^n; \mathbb{Z}_p) \to \TC(\mathcal{O}_C/p^n; \mathbb{Z}_p)$ is
an equivalence, so it suffices to show that $L_{K(1)}
\TC(\mathcal{O}_C/p^n; \mathbb{Z}_p) =0 $. 
Furthermore, according to \cite{NS18}, 
$\TC(\mathcal{O}_C/p^n;\mathbb{Z}_p)$ is an equalizer of two maps,
\begin{equation}  \label{equalizerK} \TC(\mathcal{O}_C/p^n; \mathbb{Z}_p) = \mathrm{eq}\left(\TC^-(\mathcal{O}_C/p^n; \mathbb{Z}_p) \rightrightarrows \TP(\mathcal{O}_C/p^n;
\mathbb{Z}_p)\right).\end{equation}
The first (canonical) map has cofiber given by $\Sigma^2 \THH(\mathcal{O}_C/p^n;
\mathbb{Z}_p)_{hS^1}$, which is clearly $K(1)$-acyclic as a homotopy colimit of
Eilenberg-MacLane spectra. 
Thus, $L_{K(1)} \TC^-(\mathcal{O}_C/p^n; \mathbb{Z}_p) \simeq 
L_{K(1)} \TP(\mathcal{O}_C/p^n; \mathbb{Z}_p) $, and the latter vanishes by
\Cref{K1TPvanishes}. Using the formula \eqref{equalizerK}, we get that 
$L_{K(1)} \TC(\mathcal{O}_C/p^n; \mathbb{Z}_p) = 0$ as desired. 

Now we descend to prove the result for $\mathbb{Z}/p^n$.
Let $E$ range over the finite extensions of $\mathbb{Q}_p$ inside
$\overline{\mathbb{Q}_p}$. 
For any such, we have a finite flat morphism
$\mathbb{Z}/p^n \to \mathcal{O}_E/p^n$. 
The colimit over $E$ yields $\mathcal{O}_C/p^n$.
Therefore,
in the $\infty$-category of $p$-complete $E_\infty$-rings, we have
\[ \varinjlim_E L_{K(1)} K( \mathcal{O}_E/p^n) = L_{K(1)}
K(\mathcal{O}_C/p^n). \]
Since we have just shown that the target vanishes, the source does too. Now the
source is a filtered colimit in ($p$-complete) \emph{ring} spectra, and a ring spectrum vanishes if and only if its unit is null-homotopic. We conclude that
for some finite extension $E$, $L_{K(1)} K(\mathcal{O}_E/p^n)$ vanishes. 
Finally, by the descent results of \cite{CMNN} (in particular, finite flat
descent for $L_{K(1)} K(-)$ on commutative rings), 
we find that 
$$L_{K(1)} K(\mathbb{Z}/p^n) \simeq \mathrm{Tot}\left( L_{K(1)} K(\mathcal{O}_E/p^n)
\rightrightarrows  L_{K(1)} K( \mathcal{O}_E/p^n \otimes_{\mathbb{Z}/p^n}
\mathcal{O}_E/p^n) \triplearrows \dots \right).$$
Since this is a diagram of $E_\infty$-rings, we conclude that all the terms in the totalization must vanish, and we get
$L_{K(1)} K(\mathbb{Z}/p^n) =0 $ as desired. 
\end{proof} 

\subsection{The main result for $\mathbb{Z}$-linear $\infty$-categories}

In this section, we explain the deduction of \Cref{mainthmintro}. 
This argument also appears in \cite[Sec.~3.1]{LMT}. 

Let $R$ be a commutative ring, and 
let $\mathcal{C}$ be a small $R$-linear stable $\infty$-category (always assumed
idempotent-complete). 
Given a nonzerodivisor (for simplicity) $x \in R$, 
we say that $\mathcal{C}$ is \emph{$x$-power torsion} if for each object $Y \in
\mathcal{C}$, we have that $x^n: Y \to Y$ is nullhomotopic for some $n\geq 0$. 
For instance, the 
kernel of the map $\perf(R) \to \perf(R[1/x])$, i.e., those perfect complexes
of $R$-modules which are acyclic outside of $(x)$, forms such an $R$-linear
stable $\infty$-category. 
Moreover, for each $R$-algebra $R'$ such that $R'$ is perfect as an $R$-module, we
can define the $\infty$-category 
of $R'$-modules in $\mathcal{C}$, which we denote
$\mathrm{Mod}_{R'}(\mathcal{C})$; this is then an $R'$-linear stable
$\infty$-category. 
Examples of objects in
$\mathrm{Mod}_{R'}(\mathcal{C})$ include objects of the form $R' \otimes Y$ for
$Y \in \mathcal{C}$; these are given by extension of scalars from $\mathcal{C}$. 
For simplicity, when working with these objects, we will simply write $\hom_{R'}$ instead of
$\hom_{\mathrm{Mod}_{R'}(\mathcal{C})}$. 

We use the following basic fact. 
\begin{proposition} 
\label{colimcat}
Let $\mathcal{C}$ be an $R$-linear (idempotent-complete) stable $\infty$-category which is $x$-power
torsion. Then we have, in $R$-linear stable $\infty$-categories
\begin{equation} \label{colimfun}  \varinjlim_n \mathrm{Mod}_{R/x^n}(\mathcal{C}) \simeq \mathcal{C},
\end{equation}
via the natural restriction of scalars maps. 
\end{proposition} 
\begin{proof} 
In the following, all tensor products of $R$-modules are derived. 
Let $M, N \in \mathrm{Mod}_{R/x^n}(\mathcal{C})$. 
Then for $m \geq n$, we have by adjunction
\[ \hom_{R/x^m}(M, N) = \hom_{R/x^n}(M \otimes_{R/x^m} R/x^n, N) 
= \hom_{R/x^n}(M \otimes_{R/x^n} (R/x^n \otimes_{R/x^m} R/x^{n}), N)
,\]
where the relative tensor products are regarded as $R/x^n$-modules in
$\mathrm{Ind}(\mathcal{C})$. 
Similarly, we have
$$ \hom_{R}(M, N) = \hom_{R/x^n}(M \otimes_{R/x^n} (R/x^n \otimes_{R} R/x^n),
N).$$
It therefore suffices to show that
the tower in $(R/x^n, R/x^n)$-bimodules $\left\{R/x^n \otimes_{R/x^m}
R/x^n\right\}_{m \geq n}$ is pro-constant with value  $R/x^n \otimes_R R/x^n$;
this will prove that 
\[ \hom_R(M, N) = \varinjlim_{m \geq n} \hom_{R/x^m}(M, N),  \]
and that the functor in \eqref{colimfun} is fully faithful. 
It is easy to see that any object in $\mathcal{C}$ is (at least up to
retracts) in the essential image, since generating objects $R/x \otimes Y$ are
in the essential image. 

Now the pro-constancy claim follows from the following more precise assertion: 
the tower
of simplicial commutative rings $\left\{ R/x^n \otimes_{R/x^m}
R/x^n\right\}_{m \geq n}$ is pro-constant with value 
$R/x^n \otimes_R R/x^n$. 
Indeed, $R/x^n \otimes_R R/x^n$ is the free simplicial commutative ring over
$R/x^n$ on a class in degree $1$, and a short computation shows that 
for $m \geq n$, 
$R/x^n \otimes_{R/x^m} R/x^n$ is the free simplicial commutative ring on
classes in degree $1$ and $2$; moreover, the classes in degree two form a
pro-zero system. 
\end{proof} 

Finally, we can prove \Cref{mainthmintro}, which we restate
for arbitrary $\mathbb{Z}$-linear stable $\infty$-categories.
\begin{theorem} 
\label{mainthm}
Let $\mathcal{C}$ be a $\mathbb{Z}$-linear stable $\infty$-category. Then 
$L_{K(1)} K(\mathcal{C}) = L_{K(1)} K(\mathcal{C}[1/p])$. 
\end{theorem} 
\begin{proof} 
Let $\mathcal{C}_{\mathrm{tors}} \subset \mathcal{C}$ be the subcategory of
$p$-power torsion objects. 
Then we have a localization sequence
$\mathcal{C}_{\mathrm{tors}} \to \mathcal{C} \to \mathcal{C}[1/p]$, 
so the induced sequence in algebraic $K$-theory shows that it suffices to prove
that 
$L_{K(1)} K( \mathcal{C}_{\mathrm{tors}}) = 0$. 
But we have seen (\Cref{colimcat}) that $\mathcal{C}_{\mathrm{tors}}$ is a filtered colimit of a
sequence of 
stable $\infty$-categories, each of 
which is $\mathbb{Z}/p^n$-linear for some $n$. 
By \Cref{Zpn}, $L_{K(1)} K$ vanishes
for each of these; thus, it vanishes for $\mathcal{C}_{\mathrm{tors}}$ as
desired. 
\end{proof} 

\subsection{Complements}

Combining with the main result of \cite{CMM}, we get the following.

\begin{theorem} 
\label{K1Rcomp}
Let $R$ be a commutative ring. Then there is a natural equivalence 
$L_{K(1)} \TC(R)\simeq L_{K(1)}K(R_{\widehat{p}}[1/p])$. 
If $R$ is henselian along $p$, then these are naturally equivalent to $L_{K(1)}
K(R[1/p])$. 
\end{theorem} 

\begin{remark} 

In the above statement, the $p$-completion $R_{\widehat{p}}$ can be
taken to be either derived or ordinary $p$-completion; it doesn't matter for the
statement, as the (mod $p$) $K$-theory of $\mathbb{Z}[1/p]$-algebras is
nil-invariant and truncating in the sense of \cite{LT}.  
\end{remark} 
\begin{proof} 
We claim that all of the natural maps
$$\TC(R)\rightarrow \TC(R_{\widehat{p}})\leftarrow K(R_{\widehat{p}})\rightarrow K(R_{\widehat{p}}[1/p])$$
are $K(1)$-equivalences.  For the left map, this is because $\TC/p$ is
invariant under (mod $p$) equivalences.  For the right map this is by
\Cref{mainthmintro}. For the middle map, \cite{CMM} gives a fiber square
\[ \xymatrix{
K(R_{\widehat{p}}; \mathbb{Z}_p) \ar[d]  \ar[r] &  \TC(R_{\widehat{p}}; \mathbb{Z}_p) \ar[d]  \\
K(R/p; \mathbb{Z}_p) \ar[r] &  \TC(R/p; \mathbb{Z}_p)
}.\]
But $K(1)$-localization annihilates the bottom row since $L_{K(1)}K(\mathbb{F}_p)=0$.  Thus we obtain the desired equivalence $L_{K(1)} K(R_{\widehat{p}}) \simeq L_{K(1)} \TC(R_{\widehat{p}})$.
The deduction in the $p$-henselian case follows similarly from \cite{CMM}. 
\end{proof} 

\begin{remark}This result is a form of the ``\'etale comparison theorem" of
Bhatt--Scholze in integral $p$-adic Hodge theory, \cite[Th.~9.1]{Prisms}.  Indeed, $\TC(R)$ is
closely related to the complexes $\mathbb{Z}_p(n)$ of \cite{BMS2},
whereas $L_{K(1)}K(R_{\widehat{p}}[1/p])$ is related in a similar manner to the
standard \'etale $\mathbb{Z}_p(n)$'s on the rigid analytic generic fiber
\cite{Tho85}. With respect to appropriate motivic filtrations on both sides, we
expect this result to recover the \'etale comparison theorem. 
\end{remark}

\begin{question} 
\begin{enumerate}
\item  

The first statement of \Cref{K1Rcomp} also makes sense for associative rings $R$.  It is natural to guess that the theorem holds in that greater generality, and constitutes a kind of ``non-commutative $p$-adic Hodge theory."  We remark that the only ingredient in the above proof which required commutativity was the rigidity result of \cite{CMM} for the ideal $(p)\subset R$ when $R$ is $p$-complete.

\item
One could also speculate about higher height analogs of \Cref{K1Rcomp}, in the
context of structured ring spectra $R$: is there such a thing as ``$v_n$-adic
Hodge theory"?  Note that there is a ``red shift" aspect to \Cref{K1Rcomp}, in
that $p=v_0$ is the relevant chromatic element on the inside of the $K$-theory whereas $v_1$ is the relevant chromatic element on the outside.

\end{enumerate}
\end{question}

This result can also be interpreted in the light of Selmer $K$-theory.  Recall:

\begin{definition}[Selmer $K$-theory, \cite{Artinmaps}] 
Let $\mathcal{C}$ be a $\mathbb{Z}$-linear $\infty$-category. 
We let $K^{Sel}(\mathcal{C}) = \TC(\mathcal{C}) \times_{L_{1} \TC(\mathcal{C}) }
L_1 K(\mathcal{C})$. 
\end{definition} 

As in \cite{CM}, Selmer $K$-theory, while a noncommutative invariant (i.e., one
defined for stable $\infty$-categories), turns out to recover \'etale $K$-theory
for commutative rings in degrees $\geq -1$. 
The definition of Selmer $K$-theory involves a pullback square; it is built
from three other noncommutative invariants. We observe 
here that the pullback, at least after $p$-adic completion (which we denote
by $K^{Sel}(\cdot; \mathbb{Z}_p)$) and for commutative
rings, is exactly the
arithmetic square. 

For the next result, since we need to use derived completion, we work with
connective $E_\infty$-algebras over $\mathbb{Z}$ for convenience. 
Recall also that $L_{K(1)} \colon \sp \to \sp$ is the composition of
$L_1$-localization followed by $p$-completion. 

\begin{corollary} 
Let $R$ be a connective $E_\infty$-algebra over $\mathbb{Z}$. 
Then the pullback square defining $K^{Sel}(R; \mathbb{Z}_p)$ is also the
tautological pullback square (valid for any localizing invariant)
$K^{Sel}(R_{\hat{p}}; \mathbb{Z}_p) \times_{K^{Sel}( {R}_{\hat{p}}[1/p];
\mathbb{Z}_p)} K^{Sel}(R[1/p]; \mathbb{Z}_p)$.
\label{kselpullback}
\end{corollary} 
\begin{proof} 
First, $\TC(\cdot; \mathbb{Z}_p)$ is invariant under passage to
$p$-completion and agrees with $K^{Sel}(\cdot; \mathbb{Z}_p)$ for $p$-complete
connective $E_\infty$-algebras.
Second, the factor $L_{K(1)} K(\cdot; \mathbb{Z}_p)$ is invariant under passage
to inverting $p$ on the argument (as we showed above) for
$\mathbb{Z}$-algebras and agrees with
$K^{Sel}(\cdot; \mathbb{Z}_p)$ for $\mathbb{Z}[1/p]$-algebras. Finally, the map 
$L_{K(1)} K(\cdot; \mathbb{Z}_p) \to L_{K(1)} \TC(\cdot; \mathbb{Z}_p)$
is an equivalence for $p$-complete connective $E_\infty$-algebras, thanks to
\Cref{K1Rcomp} and the Dundas--Goodwillie--McCarthy theorem. 
These three observations imply the claim. 
\end{proof} 

\begin{remark} 

\Cref{kselpullback}
raises the question whether there is a direct definition of
Selmer $K$-theory (at least after $p$-completion), without forming the above
pullback square.

\end{remark} 

\section{The K\"unneth formula}

To begin with, we recall the $K(1)$-local case of the 
celebrated result of Goerss--Hopkins--Miller \cite{GH, RezkHM}, which describes
(in this case) the
$E_\infty$-ring $KU_{\hat{p}}$ and its space of automorphisms. 
See also \cite[Sec.~5]{Ell2} for a modern account of some generalizations. 

\begin{theorem}[Goerss--Hopkins--Miller] 
\label{GHM}
The space of $E_\infty$-automorphisms of $KU_{\hat{p}}$ is given by
$\mathbb{Z}_p^{\times}$, via Adams operations $\psi^x$, $x\in\mathbb{Z}_p^\times$, characterized by $\psi^x (t)=x\cdot t$ for all $t\in\pi_2 KU_{\hat{p}}$.
\end{theorem} 

We can now state the main K\"unneth-style theorem in the commutative case.   In fact, as the proof will show, the analogous statement also holds for non-commutative rings (minus the $E_\infty$-ring structure, of course).  Closely related results appear in  \cite{DM98, Mit00} (at least at the level
of homotopy groups). 

\begin{theorem} 
\label{Kunneththm}
Let $R$ be any commutative ring. 
Then there exists a natural, $\mathbb{Z}_p^{\times}$-equivariant equivalence of
$E_\infty$-rings
\[ L_{K(1)} (K (R \otimes_{\mathbb{Z}} \mathbb{Z}[\zeta_{p^\infty}])) \simeq
L_{K(1)} (K (R) \otimes KU_{\hat{p}}),  \]
where $\mathbb{Z}_p^{\times}$ acts on 
$\mathbb{Z}[\zeta_{p^\infty}]$ by Galois automorphisms and on 
$KU_{\hat{p}}$ as in \Cref{GHM}. 
\end{theorem} 

In the above statements we are only considering $\mathbb{Z}_p^\times$ as a discrete group.  This
is for simplicity of exposition, but in fact we will also obtain the (appropriately formulated) analogous statements on the level of profinite groups, essentially as a \emph{consequence} of the statements on the level of discrete groups.  To accomplish this we will use the following lemma.  While the statement involves a non-canonical choice of $g\in\mathbb{Z}_p^\times$, in the end it will only be used to prove statements which are formulated independently of $g$.

\begin{lemma}\label{discrete2continuous}
Let $\mu$ denote the torsion subgroup of $\mathbb{Z}_p^\times$ (so $\mu=\mu_{p-1}$ for $p$ odd and $\mu=\mu_2$ for $p=2$).  Further let $g\in \mathbb{Z}_p^\times$ be an element which projects to a topological generator of the quotient $\mathbb{Z}_p^\times/\mu (\cong \mathbb{Z}_p)$, and consider the homomorphism $\mu\times \mathbb{Z}\rightarrow \mathbb{Z}_p^\times$ induced by the inclusion on the first factor and $1\mapsto g$ on the second factor.

Then the induced pullback functor
$$\pi^\ast:\operatorname{Sh}^{hyp}(B\mathbb{Z}_p^\times)\rightarrow \operatorname{PSh}(B(\mu\times \mathbb{Z}))$$
from hypercomplete sheaves of $p$-complete spectra on the site of finite continuous $\mathbb{Z}_p^\times$-sets to presheaves of $p$-complete spectra on the one-object groupoid $B(\mu\times \mathbb{Z})$ is fully faithful.  Moreover, its essential image consists of those $p$-complete spectra with $\mu\times \mathbb{Z}$-action whose (mod $p$) homotopy groups have the property that the action extends continuously to $\mathbb{Z}_p^\times$.
\end{lemma}
\begin{proof}
The pullback functor is associated to a geometric morphism of topoi, and hence
commutes with (mod $p$) homotopy group sheaves.  Thus the pullback functor
lands in the claimed full subcategory by the usual equivalence between abelian
groups sheaves on $B\mathbb{Z}_p^\times$ and abelian groups with continuous
$\mathbb{Z}_p^\times$-action.  Similarly, the pullback functor detects
equivalences, as the hypercompleteness lets us check this on (mod $p$) homotopy
group objects.  Thus it suffices to show that if $M$ is a $p$-complete spectrum
with $\mu\times \mathbb{Z}$-action whose induced action on (mod $p$) homotopy
groups extends continuously to $\mathbb{Z}_p^\times$, then $\pi^\ast \pi_\ast M\overset{\sim}{\rightarrow} M$.  This can be checked on underlying $p$-complete spectra, where it unwinds to the claim that
$$\varinjlim_H M^{h( H\cap (\mu\times \mathbb{Z}))}\rightarrow M$$
is a (mod $p$) equivalence.  Here $H$ runs over all open subgroups of
$\mathbb{Z}_p^\times$ and the superscript stands for homotopy fixed points,
compare \cite[Sec.~4.1]{CM}.  Passing to a cofinal collection of $H$'s, the above map is equivalent to
$$\varinjlim_n M^{h(p^n\mathbb{Z})}\rightarrow M.$$
Replacing $M$ by $M/p$, we may as well assume that $M$ is annihilated by a
power of $p$, in which case the condition is equivalent to demanding that the
action on the homotopy of $M$ admits a continuous extension to
$\mathbb{Z}_p^\times$, or equivalently is the union of subgroups fixed by some
$H$.  As the colimit is filtered, and the limit is uniformly finite, we can
then run a d\'evissage on the Postnikov tower of $M$ and reduce to the case where $M$ is concentrated in a single degree, which may as well be degree $0$, and there again we can assume that $M$ is fixed by all sufficiently small $H$.  It follows that the map is an equivalence in degree $0$.  In degree $1$, analyzing the colimit on the left we find that all the terms identify with $M$ but the bonding maps eventually identify with multiplication by $p$.  As $M$ is $p$-torsion the colimit gives $0$, as required. \end{proof}

We now construct the map which will implement the equivalence of 
\Cref{Kunneththm}. 
Let $\mu_{p^\infty} \subset \mathbb{Z}[\zeta_{p^\infty}]^{\times}$ be the
subgroup of $p$-power roots of unity, so $\mu_{p^\infty} \simeq
\mathbb{Q}_p/\mathbb{Z}_p.$ 
Consider the classifying space $B \mu_{p^\infty}$ as an infinite loop space; we have therefore the $E_\infty$-ring  $\Sigma^\infty_+ B \mu_{p^\infty}$.
Since $\mathbb{Z}_{p}^{\times}$ acts on $\mu_{p^\infty}$ via Galois
automorphisms, we obtain a $\mathbb{Z}_p^{\times}$-action on 
$\Sigma^\infty_+ B \mu_{p^\infty}$. 

\begin{construction}
We have a
$\mathbb{Z}_p^{\times}$-equivariant map 
of $E_\infty$-rings
\[  \psi: \left( \Sigma^\infty_+ B \mu_{p^\infty}\right)_{\hat{p}} \to K(
\mathbb{Z}[\zeta_{p^\infty}])_{\hat{p}} \]
since for any commutative ring $R$ we have a natural map $\Sigma^\infty_+
BR^{\times} \to K(R)$. 
Moreover, the source,
which is homotopy equivalent to $\left( \Sigma^\infty_+  B S^1
\right)_{\hat{p}} \simeq ( \Sigma^\infty_+ K( \mathbb{Z}_p, 2))_{\hat{p}}$, 
contains the natural Bott class $\beta \in \pi_2$, which is invariant under the
$\mathbb{Z}_p^\times$-action up to unit multiple. \end{construction}

\begin{proposition} 
$\psi$ carries $\beta$ to an invertible element in $\pi_2 L_{K(1)} ( K(
\mathbb{Z}[\zeta_{p^\infty}]))$.
\end{proposition} 
\begin{proof} 
By \'etale hyperdescent for $K(1)$-local $K$-theory \cite{Tho85},  \Cref{mainthmintro}, and
Gabber--Suslin rigidity \cite{Gabber}, it suffices to verify this
after composing to $\pi_2 L_{K(1)} ( K( k))$, where $k$ is any  
separably closed field of characteristic $\neq p$ over
$\mathbb{Z}[\zeta_{p^\infty}]$. However, this follows from Suslin's description
\cite{Suslin} of $K(k)_{\hat{p}}$ in this case. In particular, $\pi_* ( L_{K(1)} K(k)) $ 
is a Laurent polynomial algebra on $\beta$. 
\end{proof}

We use now the following fundamental result of Snaith \cite{Snaith} which gives a
description of $KU$ via the above constructions (here we only use the
$p$-complete case). 
See also \cite[Sec.~6.5]{Ell2} for a different proof. 
\begin{theorem}[Snaith] 
\label{snaiththm}
The induced map $((\Sigma^\infty_+ B\mu_{p^\infty})_{\hat{p}}[\beta^{-1}])_{\hat{p}}\rightarrow KU_{\widehat{p}}$ is an equivalence.
\end{theorem} 

This furnishes a potentially different $\mathbb{Z}_p^\times$-action on $KU_{\widehat{p}}$ from that of Theorem \ref{GHM}; but in fact it must be the same, as it does the same thing on $\pi_2$.

\begin{construction} 
We obtain a $\mathbb{Z}_p^{\times}$-equivariant map of $E_\infty$-rings
\[ KU_{\hat{p}} \to L_{K(1)} K( \mathbb{Z}[\zeta_{p^\infty}])  \]
obtained from the map $\psi$ by inverting the class $\beta$ (in the $p$-complete
sense) and using \Cref{snaiththm}. Consequently, we obtain a natural $\mathbb{Z}_p^{\times}$-equivariant map
for any $R$, 
\begin{equation} L_{K(1)} (K( R ) \otimes KU_{\hat{p}}) \to L_{K(1)} K( R \otimes_{\mathbb{Z}}
\mathbb{Z}[\zeta_{p^\infty}]). \label{Kunneth} \end{equation}

\end{construction} 

Let us pause and explain how to promote this to an equivariant map for the \emph{profinite} $\mathbb{Z}_p^\times$, formulated as in Lemma \ref{discrete2continuous}  in terms of hypercomplete sheaves on the topos $B\mathbb{Z}_p^\times$ of finite continuous $\mathbb{Z}_p^\times$ sets.

\begin{lemma}
Consider $\mu_{p^\infty}$ as equipped with its continuous action of $\mathbb{Z}_p^\times$, hence as an abelian group sheaf on $B\mathbb{Z}_p^\times$.  Thus $\Sigma^\infty_+ B\mu_{p^\infty}$ promotes to a sheaf of $E_\infty$-ring spectra on $B\mathbb{Z}_p^\times$.  Then there exists an initial $p$-complete hypercomplete sheaf of $E_\infty$-ring spectra $KU_{\hat{p}}$ on $B\mathbb{Z}_p^\times$ equipped with a map
$$\Sigma^\infty_+ B\mu_{p^\infty}\rightarrow KU_{\hat{p}}$$
satisfying the property that on underlying spectra (meaning, after pulling back to the basepoint $\ast\rightarrow B\mathbb{Z}_p^\times$) it carries $\beta$ to an invertible element.  Moreover, on underlying spectra this $KU_{\hat{p}}$ identifies with the usual $KU_{\hat{p}}$ and the map identifies with the usual one.
\end{lemma}
\begin{proof}
We choose a $g\in\mathbb{Z}_p^\times$ as in Lemma \ref{discrete2continuous} in
order to transfer this to the analogous claim for presheaves on $B(\mu\times
\mathbb{Z})$.  But then it is a consequence of the equivariance of the Snaith identification, explained above.
\end{proof}

Now we recall that $L_{K(1)} K(\mathbb{Z}[\zeta_{p^\infty}])=L_{K(1)}K(\mathbb{Z}[1/p,\zeta_{p^\infty}])$ promotes to a hypercomplete sheaf on $B\mathbb{Z}_p^\times$, by Thomason's hyperdescent theorem applied to the $p$-cyclotomic tower.  Moreover the map $B\mu_{p^\infty}\rightarrow \Omega^\infty K(\mathbb{Z}[1/p,\zeta_{p^\infty}])$ used to define $\psi$ comes from the finite level maps $B\mu_{p^n}\rightarrow  \Omega^\infty K(\mathbb{Z}[1/p,\zeta_{p^n}])$ and hence $\psi$ promotes to a map of sheaves of $E_\infty$-ring spectra on $B\mathbb{Z}_p^\times$.  Thus the above lemma does promote our naive discrete $\mathbb{Z}_p^\times$-equivariant map
$$KU_{\widehat{p}}\rightarrow L_{K(1)} K( \mathbb{Z}[\zeta_{p^\infty}])$$
to an honest one.  The claim that such a map (or one derived from it such as \eqref{Kunneth}) is an equivalence is independent of whether we think of $\mathbb{Z}_p^\times$ as a discrete or profinite group, since equivalences of hypercomplete sheaves over $B\mathbb{Z}_p^\times$ can be checked on pullback to the basepoint.

Let us formally record this more refined construction, and its fundamental property, in the following.

\begin{theorem}
\label{K1localKthy}
Let $KU_{\hat{p}}$ denote the hypercomplete $p$-complete sheaf of
$E_\infty$-algebras on $B\mathbb{Z}_p^\times$ constructed in the previous
lemma.  Let also $\pi \colon \operatorname{Spec}(\mathbb{Z}[1/p])_{et}\rightarrow B\mathbb{Z}_p^\times$ be the geometric morphism of topoi encoding the $p$-cyclotomic extension.  Then there is a natural comparison map $\pi^\ast KU_{\widehat{p}}\rightarrow L_{K(1)}K(-)$ of sheaves of $E_\infty$-rings.

Furthermore, suppose that $X$ is an algebraic space over $\mathbb{Z}[1/p]$ of finite Krull dimension with a uniform bound on the virtual (mod $p$) Galois cohomological dimension of its residue fields.  Then for $\pi_X:X_{et}\rightarrow B\mathbb{Z}_p^\times$ the composition of $\pi$ with the natural projection $X\rightarrow\operatorname{Spec}(\mathbb{Z}[1/p])$, the induced comparison map
$$\pi_X^\ast KU_{\widehat{p}}\rightarrow L_{K(1)}K(-)$$
identifies the target as the $p$-completion of the hypercompletion of the source.
\end{theorem}
\begin{proof}
The first statement was proved in the discussion just before.  For the second
statement, by Thomason's hyperdescent theorem in the general form proved in
\cite{CM}, it suffices to check this on strictly henselian local rings; by
Gabber--Suslin rigidity, we can even reduce to separably closed fields $k$.  Then this encodes the combination of Suslin's identification of $K(k)_{\widehat{p}}$ with Snaith's presentation of $KU_{\widehat{p}}$, as already explained above.
\end{proof}

When $R$ is commutative, 
one can use similar arguments to directly check that \eqref{Kunneth} is an equivalence.
However, we actually prove below a more general statement for arbitrary localizing invariants over
$\mathbb{Z}[1/p]$, which we formulate next. 
Let $R$ be a commutative $\mathbb{Z}[1/p]$-algebra and let $E$ be a localizing invariant for $R$-linear
$\infty$-categories (in the sense of \cite{BGT}) which commutes with filtered
colimits. Since everything is linear over algebraic $K$-theory, 
we obtain as well from \eqref{Kunneth} a natural $\mathbb{Z}_p^{\times}$-equivariant map 
\begin{equation}  \label{Emap} L_{K(1)} ( E(R) \otimes KU_{\hat{p}}) \to L_{K(1)} (
E(R[\zeta_{p^\infty}]))  , \end{equation}
which we will show to be an equivalence. 
 
To do this, we will need to use a type of Galois descent for the profinite
group $\mathbb{Z}_p^{\times}$; recall that $L_{K(1)} S^0 \to KU_{\hat{p}}$ is a
pro-Galois extension for the profinite group $\mathbb{Z}_p^{\times}$
in the sense studied by Rognes \cite{Rog08}. 
From this, one can obtain a type of Galois descent with respect to the
profinite group $\mathbb{Z}_p^{\times}$; here we formulate an equivalent primitive
version using the dense discrete subgroup $\mu \times \mathbb{Z}
\subset \mathbb{Z}_p^{\times}$ as in Lemma \ref{discrete2continuous}. 

First, the $\mathbb{Z}_p^{\times}$-action on $KU_{\hat{p}}$ yields a functor
\begin{equation} \label{KUfun} L_{K(1)}( \cdot \otimes KU_{\hat{p}}) :
L_{K(1)} \mathrm{Sp} \to \md_{L_{K(1)} \sp}( KU_{\hat{p}})^{h ( \mu\times \mathbb{Z})},\end{equation}

\begin{proposition} 
\label{enoughfixedpoints}
The natural functor \eqref{KUfun}
is fully faithful, and the essential image is spanned by those such that on mod
$p$ homotopy groups, the stabilizer of any element under the
$\mathbb{Z}$-action contains $p^N \mathbb{Z}$ for $N \gg 0$. 
\end{proposition} 
\begin{proof} 
Recall that $L_{K(1)} S^0 \simeq  \left( KU_{\hat{p}} \right)^{h (
\mu \times \mathbb{Z})}$. 
Therefore, the functor \eqref{KUfun} is fully faithful on the compact generator
$L_{K(1)} S^0/p$ of the source. To prove that \eqref{KUfun} is fully faithful,
it suffices to prove that  the image of this object in the target is compact. 
When $p > 2$, this follows because homotopy invariants over $\mu \times
\mathbb{Z}$ can be expressed as the retract of a finite limit. This requires an argument for $p = 2$. 
In this case, it suffices to show that the unit is compact
in
$\md(KU)^{h \{\pm 1\}}$, which follows by the Galois descent equivalence
$\md(KU)^{h \{
\pm 1\}} \simeq \md(KO)$, cf.~e.g. \cite[Th.~9.4]{MGal} for an account; in
particular, this shows that taking $\mu$-invariants here can also be expressed as a
retract of a 
finite limit.

For essential surjectivity, it suffices to show that if $M $ is a $p$-complete $KU$-module 
with compatible $\mu \times \mathbb{Z}$-action satisfying the
continuity property in the statement, then $M = 0$ if and only if $M^{h (
\mu \times \mathbb{Z})} =0 $. Indeed, suppose these homotopy
fixed points vanish. Then also by Galois descent up
the faithful $\mu$-Galois extension
of $E_\infty$-rings
$KU_{\hat{p}}^{h \mu} \to KU_{\hat{p}}$, it suffices to see
that $M^{h \mu} = 0$. Now
$(M^{h\mu})^{h\mathbb{Z}} =0$. 
But by the homotopy fixed point spectral sequence,
since $\mathbb{Z}$ has cohomological dimension $1$, we get that $M^{h
\mu} = 0$ as desired. Here we use that any
$p$-adically continuous
$\mathbb{Z}$-action on a nonzero $p$-torsion abelian group has a nontrivial fixed point. 
\end{proof} 

\begin{theorem} 
\label{locinvariantKunneth}
Let $R$ be a commutative $\mathbb{Z}[1/p]$-algebra and let $E$ be a localizing invariant on $R$-linear $\infty$-categories
which commutes with filtered colimits (or just the filtered colimit giving the $p$-cyclotomic extension of $R$). Then the natural map 
\eqref{Emap}
is an equivalence. 
\end{theorem} 
\begin{proof} 
To see that \eqref{Emap} is an equivalence, it suffices to prove that it
becomes an
equivalence
after taking $\mu \times \mathbb{Z} \subset
\mathbb{Z}_p^{\times}$-homotopy fixed points thanks to
\Cref{enoughfixedpoints}. 
The homotopy fixed points on the left-hand-side are given by $L_{K(1)} E(R)$. 
For the right-hand-side, the localizing invariant
$L_{K(1)} E(R\otimes_{\mathbb{Z}[1/p]}-)$ satisfies \'etale hyperdescent over $\mathbb{Z}[1/p]$ by \cite[Th.~7.14]{CM}.
Using the evident comparison between continuous cohomology on
$\mathbb{Z}_p^{\times}$ and discrete group cohomology on $\mu
\times \mathbb{Z}$ (which follows from Lemma \ref{discrete2continuous}), we find that the natural map 
$L_{K(1)} E(R) \to 
L_{K(1)} (
E(R[\zeta_{p^\infty}]))^{h ( \mu \times \mathbb{Z})}$ is an
equivalence. Thus, \eqref{Emap} becomes an equivalence after taking homotopy
fixed points and thus is an equivalence. 
\end{proof} 

Finally, \Cref{Kunneththm} follows, since by \Cref{mainthmintro} one reduces to
the case where $R$ is a $\mathbb{Z}[1/p]$-algebra. 
\bibliographystyle{amsalpha}
\bibliography{k1local}
\end{document}